\documentclass[11pt, a4paper]{article}
\usepackage{times}
\usepackage{a4wide}
\usepackage[dvips, hyperindex]{hyperref}
\usepackage[british]{babel}
\usepackage{enumerate, longtable}
\usepackage{amsmath, amscd, amsfonts, amsthm, amssymb, latexsym, comment, stmaryrd, graphicx}
\usepackage[all]{xy}
\usepackage[T1]{fontenc}
\usepackage[latin1]{inputenc}
\usepackage{color}

\newtheorem{thm}{Theorem}[section]
\newtheorem{lem}[thm]{Lemma}

\newtheorem{prop}[thm]{Proposition}
\newtheorem{cor}[thm]{Corollary}

\newcommand{\Hom}{{\rm Hom}}

\newcommand{\GL}{\mathrm{GL}}
\newcommand{\GSp}{\mathrm{GSp}}

\newcommand{\Sp}{\mathrm{Sp}}
\newcommand{\SL}{\mathrm{SL}}

\newcommand{\id}{\mathrm{id}}

\DeclareMathOperator{\Sym}{Sym}

\newcommand{\PGL}{\mathrm{PGL}}
\newcommand{\PSL}{\mathrm{PSL}}

\newcommand{\calL}{\mathcal{L}}

\newcommand{\calP}{\mathcal{P}}

\newcommand{\FF}{\mathbb{F}}

\newcommand{\PP}{\mathbb{P}}

\newcommand{\ZZ}{\mathbb{Z}}

\newcommand{\Fbar}{\overline{\FF}}

\newcommand{\rk}{\mathrm{rk}}

\begin{document}

\selectlanguage{british}

\title{Classification of subgroups of symplectic groups over finite fields containing a transvection}
\author{
Sara Arias-de-Reyna\footnote{Universit\'e du Luxembourg,
Facult\'e des Sciences, de la Technologie et de la Communication,
6, rue Richard Coudenhove-Kalergi,
L-1359 Luxembourg, 
Luxembourg, sara.ariasdereyna@uni.lu},
Luis Dieulefait\footnote{Departament d'\`Algebra i Geometria,
Facultat de Matem\`atiques,
Universitat de Barcelona,
Gran Via de les Corts Catalanes, 585,
08007 Barcelona, Spain, ldieulefait@ub.edu},
Gabor Wiese\footnote{Universit\'e du Luxembourg,
Facult\'e des Sciences, de la Technologie et de la Communication,
6, rue Richard Coudenhove-Kalergi,
L-1359 Luxembourg, Luxembourg, gabor.wiese@uni.lu}}
\maketitle

\begin{abstract}
In this note we give a self-contained proof of the following classification (up to conjugation)
of finite subgroups of $\GSp_n(\Fbar_\ell)$ for $\ell \ge 5$, which can be derived from
work of Kantor: $G$ is either reducible, symplectically imprimitive or it contains $\Sp_n(\FF_\ell)$.
This result is for instance useful for proving `big image' results for symplectic
Galois representations.

\end{abstract}

\section{Introduction}

In this paper we provide a self-contained proof of a classification result of subgroups of the general symplectic group over a finite field of characteristic $\ell\geq 5$ that contain a nontrivial transvection (cf.\ Theorem~\ref{thm:gp} below).

The motivation for this work came originally from Galois representations attached to automorphic forms and the applications to the inverse Galois problem. In a series of papers, we prove that for any even positive integer~$n$ and any positive integer $d$, $\mathrm{PSp}_n(\mathbb{F}_{\ell^d})$ or $\mathrm{PGSp}_n(\mathbb{F}_{\ell^d})$ occurs as a Galois group over the rational numbers for a positive density set of primes~$\ell$ (cf.\ \cite{partI}, \cite{partII}, \cite{partIII}). A key ingredient in our proof is Theorem~\ref{thm:gp}. When we were working on this project, we were not aware that this result could be obtained as a particular case of some results of Kantor~\cite{Kantor1979}, hence we worked out a complete proof, inspired by the work of Mitchell on the classification of
subgroups of classical groups. More precisely,
in an attempt to generalise Theorem~1 of~\cite{Mit} to arbitrary
dimension, one of us (S.~A.-d.-R.) came up with a precise strategy
for Theorem~\ref{thm:gp}.
Several ideas and some notation are borrowed from~\cite{LZ}.

We believe that our proof of Theorem~\ref{thm:gp} can be of independent interest, since it is self-contained and does not require any previous knowledge on linear algebraic groups beyond the basics.

In order to fix terminology, we recall some standard definitions.
Let $K$ be a field. An $n$-dimensional $K$-vector space~$V$ equipped with a symplectic form
(i.e.\ nonsingular and alternating), denoted by $\langle v,w \rangle = v \bullet w$
for $v,w \in V$, is called a {\em symplectic $K$-space}.
A $K$-subspace $W \subseteq V$ is called a {\em symplectic $K$-subspace} if the restriction of
$\langle v,w \rangle$ to $W \times W$ is nonsingular (hence, symplectic).
The {\em general symplectic group} $\GSp(V,\langle \cdot,\cdot\rangle) =: \GSp(V)$
consists of those $A \in \GL(V)$ such that there is $\alpha \in K^\times$,
the {\em multiplier} (or {\em similitude factor}) of~$A$,
such that we have $(Av) \bullet (Aw) = \alpha(v \bullet w)$ for all $v,w \in V$.
The multiplier of~$A$ is denoted by~$m(A)$.
The {\em symplectic group} $\Sp(V,\langle \cdot,\cdot\rangle) =: \Sp(V)$
is the subgroup of~$\GSp(V)$ of elements with multiplier~$1$.
An element $\tau \in \GL(V)$ is a {\em transvection} if $\tau - \id_V$ has rank~$1$,
i.e.\ if $\tau$ fixes a hyperplane pointwisely, and there is a line $U$ such that
$\tau(v)-v \in U$ for all~$v \in V$. The fixed hyperplane is called the {\em axis}
of $\tau$ and the line $U$ is the {\em centre} (or the {\em direction}).
We will consider the identity as a ``trivial transvection''.
Any transvection has determinant~$1$.
A {\em symplectic transvection} is a transvection in~$\Sp(V)$.
Any symplectic transvection has the form
$$ T_v[\lambda] \in \Sp(V): u \mapsto u + \lambda \langle u,v \rangle v$$
with {\em direction vector} $v \in V$ and {\em parameter} $\lambda \in K$
(see e.g.\ \cite{Ar}, pp.~137--138).

The main classification result of this note is the following.
A short proof, deriving it from~\cite{Kantor1979}, is contained in~\cite{partII}.

\begin{thm}\label{thm:gp}
Let $K$ be a finite field of characteristic at least~$5$ and $V$ a symplectic
$K$-vector space of dimension~$n$.
Then any subgroup~$G$ of $\GSp(V)$ which contains a nontrivial symplectic transvection
satisfies one of the following assertions:
\begin{enumerate}[1.]
\item\label{thm:gp:1} There is a proper $K$-subspace $S \subset V$ such that $G(S) = S$.
\item\label{thm:gp:2} There are nonsingular symplectic $K$-subspaces~$S_i \subset V$ with $i=1,\dots,h$
of dimension~$m$ for some $m < n$ such that
$V = \bigoplus_{i=1}^h S_i$ and for all $g \in G$ there is a permutation $\sigma_g \in \Sym_h$
(the symmetric group on $\{1,\dots,h \}$)
with $g(S_i) = S_{\sigma_g(i)}$. Moreover, the action of $G$ on the set $\{S_1, \dots, S_h\}$ thus defined is transitive.
\item\label{thm:gp:3} There is a subfield L of K such that the subgroup generated by the symplectic transvections of $G$ is conjugated (in $\GSp(V)$) to $\Sp_{n}(L)$.
\end{enumerate}
\end{thm}

\subsection*{Acknowledgements}

S.~A.-d.-R.\ worked on this article as a fellow of the Alexander-von-Humboldt foundation.
She thanks the Universit\'e du Luxembourg for its hospitality during a long term visit in 2011.
She was also partially supported by the project MTM2012-33830 of the Ministerio de Econom\'ia y Competitividad of Spain.
L.~V.~D. was supported by the project MTM2012-33830 of the Ministerio de Econom\'ia y Competitividad of Spain and by an ICREA Academia Research Prize.
G.~W.\ was partially supported by the DFG collaborative research centre TRR~45,
the DFG priority program~1489 and the Fonds National de la Recherche Luxembourg (INTER/DFG/12/10).
S.~A.-d.-R. and G.~W. thank the Centre de Recerca Matem\`{a}tica  for its support and hospitality during a long term visit in 2010.

The authors thank the anonymous referee of~\cite{partII} and Gunter Malle for suggesting
the alternative proof of Theorem~\ref{thm:gp} based Kantor's paper~\cite{Kantor1979}, which is
given in~\cite{partII}.

\section{Symplectic transvections in subgroups}

Recall that the full symplectic group is generated by all its transvections.
The main idea in this part is to identify the subgroups of the general symplectic group
containing a transvection by the centres of the transvections in the subgroup.

Let $K$ be a finite field of characteristic~$\ell$ and $V$ a symplectic $K$-vector space
of dimension~$n$. Let $G$ be a subgroup of~$\GSp(V)$.
A main difficulty in this part stems from the fact that $K$ need not be a prime field,
whence the set of direction vectors of the transvections contained in~$G$
need not be a $K$-vector space.
Suppose, for example, that we want to deal with the subgroup
$G=\Sp_{n}(L)$ of $\Sp_{n}(K)$ for $L$ a subfield of~$K$. Then the
directions of the transvections of~$G$ form the $L$-vector space $L^{n}$
contained in~$K^{n}$. It is this what we have in mind when we introduce the
term $(L,G)$-rational subspace below. In order to do so, we set up some more notation.

Write $\calL(G)$ for the set of $0 \neq v \in V$ such that $T_v[\lambda] \in G$ for some
$\lambda \in K$. More naturally, this set should be considered
as a subset of $\PP(V)$, the projective space consisting of the lines in~$V$.
We call it the {\em set of centres (or directions) of the symplectic transvections in~$G$}.
For a given nonzero vector $v \in V$, define the {\em parameter group of direction~$v$ in~$G$} as
$$\calP_v(G) := \{ \lambda \in K \;|\; T_v[\lambda] \in G \}.$$
The fact that $T_v(\mu) \circ T_v(\lambda) = T_v(\mu + \lambda)$ shows that $\calP_v(G)$
is a subgroup of the additive group of~$K$.
If $K$ is a finite field of characteristic~$\ell$, then $\calP_v(G)$ is a finite
direct product of copies of $\ZZ/\ell\ZZ$. Denote the number of factors by $\rk_v(G)$.
Because of $\calP_{\lambda v}(G) = \frac{1}{\lambda^2} \calP_v(G)$ for $\lambda \in K^\times$,
it only depends on the centre $U := \langle v \rangle_K \in \calL(G) \subseteq \PP(V)$,
and we call it the {\em rank of~$U$ in~$G$}, although we will not make use of this in our argument.

We find it useful to consider the surjective map
$$ \Phi: V \times K \xrightarrow{(v,\lambda) \mapsto T_v[\lambda]}
\{\text{symplectic transvections in } \Sp(V)\}.$$
The multiplicative group $K^\times$ acts on $V \times K$ via
$x (v,\lambda) := (xv, x^{-2} \lambda)$. Passing to the quotient modulo this action yields
a bijection
$$ (V\setminus\{0\} \times K)/K^\times \xrightarrow{(v,\lambda) \mapsto T_v[\lambda]}
\{\text{nontrivial symplectic transvections in } \Sp(V)\}.$$
When we consider the first projection $\pi_V: V \times K \twoheadrightarrow V$
modulo the action of $K^\times$ we obtain
$$ \pi_V: (V\setminus\{0\} \times K)/K^\times \twoheadrightarrow \PP(V),$$
which corresponds to sending a nontrivial transvection to its centre.
Let $W$ be a $K$-subspace of~$V$. Then $\Phi$ gives a bijection
$$ (W \setminus \{0\}\times K)/K^\times \xrightarrow{(v,\lambda) \mapsto T_v[\lambda]}
\{ \text{nontrivial symplectic transvections in } \Sp(V) \text{ with centre in } W\}.$$
Let $L$ be a subfield of~$K$.
We call an $L$-vector space $W_L \subseteq V$ {\em $L$-rational}
if $\dim_K W_K = \dim_L W_L$ with $W_K := \langle W_L \rangle_K$
and $\langle \cdot,\cdot \rangle$ restricted to $W_L \times W_L$ takes values in~$L$.
An $L$-vector space $W_L \subseteq V$ is called {\em $(L,G)$-rational}
if $W_L$ is $L$-rational and $\Phi$ induces a bijection
$$ (W_L\setminus \{0\} \times L)/L^\times \xrightarrow{(v,\lambda) \mapsto T_v[\lambda]}
 G \cap \{ \text{nontrivial sympl. transvections in } \Sp(V) \text{ with centre in } W_K\}.$$
Note that $(W_L \setminus\{0\} \times L)/L^\times$ is naturally a subset of
$(W_K  \setminus\{0\} \times K)/K^\times$.
A $K$-subspace $W \subseteq V$ is called {\em $(L,G)$-rationalisable} if there exists
an $(L,G)$-rational $W_L$ with $W_K = W$.
We speak of an $(L,G)$-rational symplectic subspace~$W_L$ if it is $(L,G)$-rational and
symplectic in the sense that the restricted pairing is non-degenerate on~$W_L$.
Let $H_L$ and $I_L$ be two $(L,G)$-rational symplectic subspaces of~$V$.
We say that {\em $H_L$ and $I_L$ are $(L,G)$-linked} if there is
$0 \neq h \in H_L$ and $0 \neq w \in I_L$ such that $h+w \in \calL(G)$.

\section{Strategy}

Now that we have set up all notation, we will describe the strategy behind the proof
of Theorem~\ref{thm:gp}, as a service for the reader.

If one is not in case~\ref{thm:gp:1}, then there are `many' transvections in~$G$, as otherwise
the $K$-span of $\calL(G)$ would be a proper subspace of~$V$ stabilised by~$G$.
The presence of `many' transvection is used first in order to show the
existence of a subfield $L \subseteq K$ and an $(L,G)$-rational symplectic
plane $H_L \subseteq V$.
For this it is necessary to replace $G$ by one of its conjugates inside $\GSp(V)$.
The main ingredient for the existence of $(L,G)$-rational symplectic planes,
which is treated in Section~\ref{sec:existenceLG}, is Dickson's classification of the finite
subgroups of $\PGL_2(\Fbar_\ell)$.

The next main step is to show that two $(L,G)$-linked symplectic spaces in~$V$
can be merged into a single one. This is the main result of Section~\ref{sec:merge}.
The main input is a result of Wagner for transvections in three dimensional vector spaces,
proved in Appendix~\ref{appendix:A}.

The merging results are applied to extend the $(L,G)$-rational symplectic plane further,
using again the existence of `many' transvections. We obtain a maximal $(L,G)$-rational
symplectic space $I_L \subseteq V$ in the sense that $\calL(G) \subset I_K \cup I_K^\perp$,
which is proved in Section~\ref{sec:extend}.
The proof of Theorem~\ref{thm:gp} can be deduced from this (see Section~\ref{sec:proof})
because either $I_K$ equals $V$, that is the huge image case, or translating $I_K$
by elements of~$G$ gives the decomposition in case~\ref{thm:gp:2}.

\section{Simple properties}

We use the notation from the Introduction.
In this subsection we list some simple lemmas illustrating and characterising the definitions
made above.

\begin{lem}\label{lem:rational-lines}
Let $v \in \calL(G)$. Then $\langle v \rangle_L$ is an $(L,G)$-rational line
if and only if $\calP_v(G) = L$.
\end{lem}

\begin{proof}
This follows immediately from that fact that all transvections with centre
$\langle v \rangle_K$ can be written uniquely as $T_v[\lambda]$ for some $\lambda \in K$.
\end{proof}

\begin{lem}\label{lem:rational-subs}
Let $W_L \subseteq V$ be an $(L,G)$-rational space and $U_L$ an $L$-vector subspace of~$W_L$.
Then $U_L$ is also $(L,G)$-rational.
\end{lem}

\begin{proof}
We first give two general statements about $L$-rational subspaces. Let $u_1,\dots,u_d$
be an $L$-basis of~$U_L$ and extend it by $w_1,\dots,w_e$ to an $L$-basis of~$W_L$.
As $W_L$ is $L$-rational, the chosen vectors remain linearly independent over~$K$,
and, hence, $U_L$ is $L$-rational. Moreover, we see, e.g.\ by writing
down elements in the chosen basis, that $W_L \cap U_K = U_L$.

It is clear that $\Phi$ sends elements in $(U_L \times L)/L^\times$
to symplectic transvections in~$G$ with centres in~$U_K$.
Conversely, let $T_v[\lambda]$ be such a transvection. As $W_L$ is $(L,G)$-rational,
$T_v[\lambda] = T_u[\mu]$ with some $u \in W_L$ and $\mu \in L$. Due to
$W_L \cap U_K = U_L$, we have $u \in U_L$ and the tuple $(u,\mu)$ lies in $U_L \times L$.
\end{proof}

\begin{lem}\label{lem:definition}
Let $W_L \subseteq V$ be an $L$-rational subspace of~$V$. Then the following assertions are equivalent:
\begin{enumerate}[(i)]
\item\label{lem:definition:i} $W_L$ is $(L,G)$-rational.
\item\label{lem:definition:ii}
\begin{enumerate}[(a)]
\item\label{lem:definition:ii:a} $T_{W_L}[L] := \{ T_v[\lambda]\;|\; \lambda \in L, \; v \in W_L \} \subseteq G$ and
\item\label{lem:definition:ii:b} for each $U \in \calL(G) \subseteq \PP(V)$ with $U \subseteq W_K$ there is a
$u\in U \cap W_L$ such that $\calP_u(G) = L$
(i.e.\ $\langle u \rangle_L$ is an $(L,G)$-rational line contained in~$U$ by Lemma~\ref{lem:rational-lines}).
\end{enumerate}
\end{enumerate}
\end{lem}

\begin{proof}
'(\ref{lem:definition:i}) $\Rightarrow$ (\ref{lem:definition:ii}):'
Note that (\ref{lem:definition:ii:a}) is clear. For (\ref{lem:definition:ii:b}),
let $U \in \calL(G)$ with $U \subseteq W_K$.
Hence, there is $u \in U$ and $\lambda \in K^\times$ with $T_u[\lambda] \in G$.
As $W_L$ is $(L,G)$-rational, we may assume that $u \in W_L$ and $\lambda \in L$.
Lemma~\ref{lem:rational-subs} implies that $\langle u \rangle_L$ is
an $(L,G)$-rational line.

'(\ref{lem:definition:ii}) $\Rightarrow$ (\ref{lem:definition:i}):'
Denote by $\iota$ the injection
$(W_L \setminus\{0\} \times L)/L^\times \hookrightarrow (W_K \setminus\{0\} \times K)/K^\times$.
By (\ref{lem:definition:ii:a}), the image of $\Phi \circ \iota$ lies in~$G$. It remains to prove
the surjectivity of this map onto the symplectic transvections of~$G$ with centres in~$W_K$.
Let $T_v[\lambda]$ be one such. Take $U = \langle v \rangle_K$.
By (\ref{lem:definition:ii:b}), there is $v_0 \in U$ such that $U_L = \langle v_0 \rangle_L \subseteq W_L$
is an $(L,G)$-rational line. In particular, $T_v[\lambda] = T_{v_0}[\mu]$
with some $\mu \in L$, finishing the proof.
\end{proof}

\begin{lem}\label{lem:basechange}
Let $A \in \GSp(V)$ with multiplier~$\alpha \in K^\times$.
Then $A T_v[\lambda] A^{-1} = T_{Av}[\frac{\lambda}{\alpha}]$.
In particular, the notion of $(L,G)$-rationality is not stable under conjugation.
\end{lem}

\begin{proof}
For all $w\in V$, $AT_v[\lambda]A^{-1}(w)=A(A^{-1}w + \lambda(A^{-1}w \bullet v) v)=w + \lambda (
A^{-1}w\bullet v) Av$. Since $A$ has multiplier $\alpha$, $w\bullet Av=\alpha(A^{-1}w\bullet v)$, hence
$AT_v[\lambda]A^{-1}(w)=w + \frac{\lambda}{\alpha} (w\bullet Av) Av=T_{Av}[\frac{\lambda}{\alpha}](w)$.
\end{proof}

\begin{lem}\label{lem:fix}
The group $G$ maps $\calL(G)$ into itself.
\end{lem}

\begin{proof}
Let $g \in G$ and $w \in \calL(G)$, say $T_w[\lambda] \in G$. Then
by Lemma~\ref{lem:basechange} we have
$g T_w[\lambda] g^{-1} = T_{gw}[\frac{\lambda}{\alpha}]$, where
$\alpha$ is the multiplier of~$g$. Hence, $g(w) \in \calL(G)$.
\end{proof}

The following lemma shows that the natural projection yields a bijection between
transvections in the symplectic group and their images in the projective symplectic group.

\begin{lem}\label{lem:auxiliar}
Let $V$ be a symplectic $K$-vector space, $0 \neq u_1,u_2 \in V$.
If $T_{u_1}[\lambda_1]^{-1}T_{u_2}[\lambda_2]\in \{a\cdot \mathrm{Id}: a\in K^\times\}$,
then $T_{u_1}[\lambda_1]=T_{u_2}[\lambda_2]$.
\end{lem}

\begin{proof}
Assume $T_{u_1}[\lambda_1]^{-1}T_{u_2}[\lambda_2]=a\mathrm{Id}$. Then for all $v\in V$, $T_{u_2}[\lambda_2](v)-T_{u_1}[\lambda_1](av)=0$. In particular, taking $v=u_1$,
$T_{u_2}[\lambda_2](u_1)-T_{u_1}[\lambda_1](au_1)=u_1 + \lambda_2(u_1\bullet u_2)u_2-au_1=0$,
hence either $u_1$ and $u_2$ are linearly dependent or $a=1$ (thus both transvections coincide).
Assume then that $u_2=bu_1$ for some $b\in K^\times$.
Then for all $v\in V$ we have
$T_{bu_1}[\lambda_2](v)-T_{u_1}[\lambda_1](av)= v + \lambda_2 b^2 (v \bullet u_1) u_1 - av - \lambda_1 a(v\bullet u_1)u_1=(a-1)v + (\lambda_2 b^2 - a\lambda_1)(v\bullet u_1)u_1=0$.
Choosing $v$ linearly independent from $u_1$, we obtain $a=1$, as we wished to prove.
\end{proof}

\section{Existence of $(L,G)$-rational symplectic planes}\label{sec:existenceLG}

Let, as before, $K$ be a finite field of characteristic~$\ell$, $V$ a $n$-dimensional
symplectic $K$-vector space and $G \subseteq \GSp(V)$ a subgroup.
We will now prove the existence of $(L,G)$-rational symplectic planes if there are two
transvections in~$G$ with nonorthogonal directions.

Note that any additive subgroup $H \subseteq K$ can appear as a parameter group
of a direction. Just take $G$ to be the subgroup of $\GSp(V)$ generated by the transvections in
one fixed direction with parameters in~$H$.
It might seem surprising that the existence of two nonorthogonal centres forces
the parameter group to be the additive group of a subfield $L$ of~$K$
(up to multiplication by a fixed scalar).
This is the contents of Proposition~\ref{prop:mitchell},
which is one of the main ingredients for this article.
This proposition, in turn, is based on Proposition~\ref{prop:1911},
going back to Mitchell (cf.\ \cite{Mitchell1911}).
To make this exposition self-contained we also include a proof of it,
which essentially relies on Dickson's classification of the finite subgroups
of~$\PGL_2(\Fbar_\ell)$.
Recall that an {\em elation} is the image in $\PGL(V)$ of a transvection in $\GL(V)$.

\begin{prop}\label{prop:1911}
Let $V$ be a 2-dimensional $K$-vector space with basis $\{e_1, e_2\}$
and $\Gamma \subseteq \PGL(V)$ a subgroup that
contains two nontrivial elations whose centers $U_1$ and $U_2$ are different.
Let $\ell^m$ be the order of an $\ell$-Sylow subgroup of~$\Gamma$.

Then $K$ contains a subfield~$L$ with $\ell^m$ elements.
Moreover, there exists $A\in \PGL_2(K)$ such that
$AU_1=\langle e_1\rangle_K$, $AU_2=\langle e_2\rangle_K$,
and $A\Gamma A^{-1}$ is either $\PGL(V_L)$ or $\PSL(V_L)$,
where $V_L=\langle e_1, e_2\rangle_L$.
\end{prop}

\begin{proof}
Since there are two elations $\tau_1$ and $\tau_2$ with independent directions $U_1$ and~$U_2$,
Dickson's classification of subgroups of $\PGL_2(\Fbar_\ell)$ (Section 260 of~\cite{Dickson})
implies that there is $B\in \PGL_2(K)$ such that $B\Gamma B^{-1}$ is either $\PGL(V_L)$ or $\PSL(V_L)$,
where $L$ is a subfield of~$K$ with~$\ell^m$ elements. By Lemma~\ref{lem:basechange}, the direction of
$B\tau_i B^{-1}$ is $BU_i$ for~$i=1,2$ and the lines $BU_i$ are of the form $\langle d_i \rangle_K$
with $d_i \in V_L$ for $i=1,2$.
As $\PSL(V_L)$ acts transitively on~$V_L$, there is $C \in \PSL(V_L)$
such that $CU_1=\langle e_1 \rangle_K$ and $CU_2=\langle e_2 \rangle_K$.
Setting $A := CB$ yields the proposition.
\end{proof}

Although the preceding proposition is quite simple, the very important consequence it has
is that the conjugated elations $A\tau_iA^{-1}$ both have direction vectors that can be defined
over the same $L$-rational plane.

\begin{lem}\label{lem:1911}
Let $V$ be a $2$-dimensional $K$-vector space, $G \subseteq \GL(V)$ containing two
transvections with linearly independent directions $U_1$ and~$U_2$.
Let $\ell^m$ be the order of any $\ell$-Sylow subgroup of~$G$.

Then $K$ contains a subfield $L$ with $\ell^m$ elements and there are
$A \in \GL(V)$ and an $(L,AGA^{-1})$-rational plane $V_L \subseteq V$.
Moreover, $A$ can be chosen such that $AU_i = U_i$ for $i=1,2$. Furthermore,
if $u_1 \in U_1$ and $u_2 \in U_2$ are such that $u_1 \bullet u_2 \in L^\times$,
then $V_L$ can be chosen to be $\langle u_1,u_2\rangle_L$.
\end{lem}

\begin{proof}
We apply Proposition~\ref{prop:1911} with $e_1=u_1$, $e_2=u_2$, and $\Gamma$ the image of $G$ in $\PGL(V)$,
and obtain $A \in \GL(V)$ (any lift of the matrix provided by the proposition)
such that $A \Gamma A^{-1}$ equals $\PSL(V_L)$ or $\PGL(V_L)$ for the
$L$-rational plane $V_L=\langle u_1, u_2\rangle_L \subseteq V$, and $AU_i = U_i$ for $i=1,2$.
For $\PSL(V_L)$ and $\PGL(V_L)$ it is true that the elations contained in them are
precisely the images of $T_v[\lambda]$ for $v \in V_L$ and $\lambda \in L$.

First, we know that all such $T_v[\lambda]$ are contained in $\SL(V_L)$ and,
thus, in $AGA^{-1}$ (since $A\Gamma A^{-1}$ is $\PSL(V_L)$ or $\PGL(V_L)$).
Second, by Lemma~\ref{lem:auxiliar} the image of $T_v[\lambda]$ in~$A\Gamma A^{-1}$
has a unique lift to a transvection in~$\SL(V_L) \subseteq AGA^{-1}$,
namely $T_v[\lambda]$.
This proves that the transvections of $AGA^{-1}$ are precisely the
$T_v[\lambda]$ for $v \in V_L$ and $\lambda \in L$. Hence, $V_L$ is
an $(L,AGA^{-1})$-rational plane.
\end{proof}

\begin{lem}\label{lem:extend}
Let $U_1, U_2 \in \calL(G)$ be such that $H = U_1 \oplus U_2$
is a symplectic plane in~$V$.
By $G_0$ we denote the subgroup $\{ g \in G \;|\; g(H) \subseteq H\}$ and by
$G|_H$ the restrictions of the elements of~$G_0$ to~$H$.

Then $\calL(G|_H) \subseteq \calL(G)$ (under the inclusion $\PP(H) \subseteq \PP(V)$).
\end{lem}

\begin{proof}
Let $\tau_i \in G$ be transvections with directions $U_i$ for $i=1,2$.
Clearly, $\tau_1, \tau_2 \in G_0$ and their restrictions to~$H$ are symplectic
transvections with the same directions.
Consequently, Lemma~\ref{lem:1911} provides us with $A \in \GL(H)$
and an $(L,AGA^{-1})$-rational plane $H_L \subseteq H$.

Let $U \in \calL(G|_H)$. This means that there is $g \in G_0$ such that
$g|_H$ is a transvection with direction~$U$, so that $A g|_H A^{-1}$ is a transvection in
$A G|_H A^{-1}$ with direction~$AU$ by Lemma~\ref{lem:basechange}.
As $H_L$ is $(L,AG|_H A^{-1})$-rational, all transvections $T_v[\lambda]$
for $v \in H_L$ and $\lambda \in L$ lie in~$AG|_H A^{-1}$,
whence $AG|_H A^{-1}$ contains~$\SL(H_L)$.
Consequently, there is $h \in AG|_H A^{-1}$ such that $h AU = AU_1$.
But $A^{-1} h A \in G|_H$, whence there is $\gamma \in G_0$ with
restriction to~$H$ equal to $A^{-1} h A$.
As $\gamma H \subseteq H$, it follows that
$\gamma U = \gamma|_H U = A^{-1} h A U = U_1$.
Now, $\gamma^{-1} \tau_1 \gamma$ is a transvection in~$G$ with centre $\gamma^{-1}U_1 = U$,
showing $U \in \calL(G)$.
\end{proof}

\begin{cor}\label{cor:extend} Let $U_1, U_2 \in \calL(G)$ be such that $H = U_1 \oplus U_2$
is a symplectic plane in~$V$.
By $G_0$ we denote the subgroup $\{ g \in G \;|\; g(H) \subseteq H\}$ and by
$G|_H$ the restrictions of the elements of~$G_0$ to~$H$. Then the transvections of $G\vert_H$ are the restrictions to $H$ of the transvections of $G$ with centre in $H$.
\end{cor}

\begin{proof}
Let $T$ be the subgroup of $G$ generated by the transvections of $G$ with centre in $H$. We can naturally identify $T$ with $T\vert_H$. Let $U$ be the subgroup of $G\vert_H$ generated by the transvections of $G\vert_H$. We have that $T\vert_H\subset U$.

Applying Lemma \ref{lem:1911} to the $K$-vector space $H$ and the subgroup $U\subset \GL(H)$, there exists a subfield $L\subset K$, and an $L$-rational plane $H_L$ such that $U$ is conjugate to $\SL(H_L)$, hence $U\simeq \SL_2(L)$. Applying Lemma \ref{lem:1911} to the $K$-vector space $H$ and the subgroup $T\vert_H$, we obtain a subfield $L'\subset K$, and an $L'$-rational plane $H_{L'}$ such that $T\vert_H$ is conjugate to $\SL(H_{L'})$, hence $H\simeq \SL_2(L')$. But $\calL(T\vert_H)=\calL(G)\cap H=\calL(G\vert_H)=\calL(U)$ by Lemma~\ref{lem:extend}, whence $L=L'$ and the cardinalities of $U$ and $T\vert_H$ coincide. Therefore they are equal.
\end{proof}

\begin{prop}\label{prop:mitchell}
Let $U_1, U_2 \in \calL(G) \subseteq \PP(V)$ which are not orthogonal.
Then there exist a subfield $L \le K$, $A \in \GSp(V)$, and an
$L$-rational symplectic plane~$H_L$ such that
$AU_1 \subseteq H_K$, $AU_2 \subseteq H_K$ and such that $H_L$ is $(L,AGA^{-1})$-rational.
Moreover, if we fix $u_1\in U_1$, $u_2\in U_2$ such that $u_1\bullet u_2\in L^\times$,
we can choose $H_L=\langle u_1, u_2\rangle_L$ and $A$ satisfying $AU_1= U_1$, $AU_2=U_2$.
\end{prop}

\begin{proof}
Let $H = U_1 \oplus U_2$ and note that this is a symplectic plane.
Define $G_0$ and $G|_H$ as in Lemma~\ref{lem:extend}.
Lemma~\ref{lem:1911} provides us with
$B \in \GL(H)$ such that $BU_i=U_i$ for $i=1,2$ and such that
$H_L = \langle u_1, u_2 \rangle_L$ is an $(L,B G|_H B^{-1})$-rational plane.
We choose $A \in \GSp(V)$ such that $A H \subseteq H$ and $A|_H = B$
(this is possible as any symplectic basis of $H$ can be extended to a symplectic
basis of~$V$).
We want to prove that $H_L$ is an $(L,AGA^{-1})$-rational symplectic plane in~$V$.

And, indeed, by Corollary~\ref{cor:extend}, the nontrivial transvections of $AGA^{-1}$ with direction in $H$ coincide with the nontrivial transvections of $BG\vert_HB^{-1}$, which in turn correspond bijectively to $(H_L\setminus \{0\}\times L)/L$.
\end{proof}

Note that Theorem~\ref{thm:gp} is independent of conjugating $G$ inside~$\Sp(V)$.
Hence, we will henceforth work with $(L,G)$-rational symplectic spaces
(instead of $(L,AGA^{-1})$-rational ones).

\begin{cor}\label{cor:mitchell}
\begin{enumerate}[(a)]
\item\label{cor:mitchell:a} Let $H_L$ be an $L$-rational
plane which contains an $(L,G)$-rational line $U_{1,L}$
as well as an $L$-rational line $U_{2,L}$ not orthogonal to~$U_{1,L}$ with
$U_{2,K} \in \calL(G)$.

Then $H_L$ is an $(L,G)$-rational symplectic plane.

\item\label{cor:mitchell:b} Let $U_{1,L} = \langle u_1 \rangle_L$ be an $(L,G)$-rational line
and $U_2 = \langle u_2 \rangle_K \in \calL(G)$ such that $u_1 \bullet u_2 \in L^\times$.

Then $\langle u_1, u_2 \rangle_L$ is an $(L,G)$-rational symplectic plane.
\end{enumerate}
\end{cor}

\begin{proof}
(\ref{cor:mitchell:a}) Fix $u_1\in U_{1, L}$ and $u_2\in U_{2, L}$ such that $u_1\bullet u_2=1$,
and call $W_L=\langle u_1, u_2\rangle_L$.
Apply Proposition \ref{prop:mitchell}: we get $L\subseteq K$ and $A\in \GSp(V)$
such that $\langle AU_{1, L}\rangle_K=\langle u_1\rangle_K$, $AU_2=\langle u_2\rangle_K$
and $W_L$ is $(L, AGA^{-1})$-rational.
Let $a_1, a_2\in K^\times$ be such that $Au_1=a_1u_1$ and $Au_2=a_2u_2$.
The proof will follow three steps: we will first see that $\calP_{u_2}(G)=L$,
then we will see that $H_L$ satisfies Lemma \ref{lem:definition}~(\ref{lem:definition:ii:a}) and
finally we will see that $H_L$ satisfies Lemma \ref{lem:definition}~(\ref{lem:definition:ii:b}).

Let $\alpha$ be the multiplier of~$A$.
First note the following equality between $\alpha$, $a_1$ and $a_2$:
$$1=u_1\bullet u_2=\frac{1}{\alpha} (Au_1\bullet Au_2)
   = \frac{1}{\alpha} (a_1 u_1\bullet a_2u_2)=\frac{a_1a_2}{\alpha}.$$
Recall that $\calP_{av}(G)=\frac{1}{a^2}\calP_v(G)$, and,
from Lemma \ref{lem:basechange} it follows that
$\calP_{Av}(AGA^{-1})=\frac{1}{\alpha}\calP_v(G)$.

On the one hand, since $U_{1, L}$ is $(L, G)$-rational and $u_1\in U_{1, L}$,
we know that $\calP_{u_1}(G)=L$ by Lemma~\ref{lem:rational-lines}.
On the other hand, since $\langle u_1\rangle_L$ is $(L, AGA^{-1})$-rational,
$\calP_{u_1}(AGA^{-1})=L$, hence $\calP_{u_1}(G)=\frac{\alpha}{a_1^2}L$. We thus have $\frac{\alpha}{a_1^2}\in L$.
Moreover, since $\langle u_2\rangle_L$ is $(L, AGA^{-1})$-rational
(e.g.\ using Lemma~\ref{lem:rational-subs}), we have that $\calP_{u_2}(AGA^{-1})=L$, hence
$\calP_{u_2}(G)= \frac{\alpha}{a_2^2}L =\frac{a_1^2\alpha}{\alpha^2}L=\frac{a_1^2}{\alpha}L=L$.
This proves that $\langle u_2 \rangle_L$ is $(L, G)$-rational by Lemma~\ref{lem:rational-lines}.

Next we will see that $T_{H_L}[L]\subseteq G$.
Let $b_1, b_2\in L$ with $b_1 \neq 0$ and $\lambda\in L^\times$.
Consider the transvection $T_{b_1 u_1 + b_2 u_2}[\lambda]$.
We want to prove that it belongs to $G$. We compute
\begin{equation*}AT_{b_1 u_1 + b_2 u_2}[\lambda]A^{-1}= T_{A(b_1 u_1 + b_2 u_2)}[\frac{\lambda}{\alpha}]=
T_{b_1a_1 u_1 + b_2 a_2u_2}[\frac{\lambda}{\alpha}]=T_{u_1 + \frac{b_2 a_2}{b_1a_1}u_2}[\frac{b_1^2a_1^2\lambda}{\alpha}].
\end{equation*}
Note that since $\frac{a_1}{a_2}=\frac{a_1^2}{\alpha}\in L$ and
since $W_L=\langle u_1, u_2\rangle_L$ is $(L, AGA^{-1})$-rational,
it follows that $AT_{b_1 u_1 + b_2 u_2}[\lambda]A^{-1}\in AGA^{-1}$,
and therefore $T_{b_1 u_1 + b_2 u_2}[\lambda]\in G$.
Note that the same conclusion is valid for $b_1=0$ as $\langle u_2 \rangle_L$ is $(L, G)$-rational.

Finally it remains to see that if $U\in \calL(G)\cap \langle H_L \rangle_K$,
then there is $u\in U\cap H_L$ with $\calP_u(G)=L$.
Assume that $U\in \calL(G)\cap \langle H_L \rangle_K$. Since we have seen that
$\langle u_2\rangle_L$ is $(L, G)$-rational, we can assume that $U\not=\langle u_2\rangle_K$.
Therefore we can choose an element $v\in U$ with $v=u_1 + bu_2$, for some $b\in K$.
It suffices to show that $b\in L$. Let $T_v[\lambda]\in G$ be a transvection with direction $U$.
Then computing $AT_v[\lambda]A^{-1}$ as above, we get that
$AT_v[\lambda]A^{-1}=T_{u_1 + \frac{b a_2}{a_1}u_2}[\frac{a_1^2\lambda}{\alpha}]$
is a transvection with direction in $\calL(AGA^{-1})\cap W_L$,
hence the $(L, AGA^{-1})$-rationality of $W_L$ implies that $b\in L$.

(\ref{cor:mitchell:b}) follows from~(\ref{cor:mitchell:a})
by observing that the condition $u_1 \bullet u_2 \in L^\times$
ensures that $\langle u_1, u_2 \rangle_L$ is an $L$-rational symplectic plane.
\end{proof}

The next corollary says that the translate of each vector in an $(L,G)$-rational symplectic
space by some orthogonal vector~$w$ is the centre of a transvection if this is the case
for one of them.

\begin{cor}\label{cor:translate}
Let $H_L \subseteq V$ be an $(L,G)$-rational symplectic space. Let $w \in H_K^\perp$ and
$0 \neq h \in H_L$ such that $\langle h+w \rangle_K \in \calL(G)$.
Then $\langle h_1 + w \rangle_L$ is an $(L,G)$-rational line for all $0 \neq h_1 \in H_L$.
\end{cor}

\begin{proof}
Assume first that $H_L$ is a plane. Let $\hat{h} \in H_L$ with $\hat{h} \bullet h = 1$ (hence $H_L=\langle h, \hat{h}\rangle_L$). As $\langle \hat{h} \rangle_L$
is an $(L,G)$-rational line and $\hat{h} \bullet (h+w) = 1$, it follows that
$\langle \hat{h}, h+w \rangle_L$ is an $(L,G)$-rational plane by Corollary~\ref{cor:mitchell}.
Consequently, for all $\mu \in L$ we have that $\langle \mu \hat{h} + h + w\rangle_L$
is an $(L,G)$-rational line.
Let now $\mu \in L^\times$. Then $(\mu \hat{h} + h + w) \bullet h = \mu \neq 0$,
whence again by Corollary~\ref{cor:mitchell} $\langle \mu \hat{h} + h + w, h \rangle_L$
is an $(L,G)$-rational plane. Thus, for all $\nu \in L$ it follows that
$\langle \mu \hat{h} + (\nu+1) h + w \rangle_L$ is an $(L,G)$-rational line.
In order to get rid of the condition $\mu \neq 0$, we exchange the roles of $h$ and~$\hat{h}$,
yielding the statement for planes.

To extend it to any symplectic space $H_L$, note that, if $h_1, h_2\in H_L$ are nonzero elements,
there exists an element $\hat{h}\in H_L$ such that $h_1\bullet \hat{h}\not=0$,
$h_2\bullet \hat{h}\not=0$. Namely, let $\hat{h}_1, \hat{h}_2$ be such that
$h_1\bullet \hat{h}_1\not=0$, $h_2\bullet \hat{h}_2\not=0$
(they exist because on $H_L$ the symplectic pairing is nondegenerate).
If $h_2\bullet \hat{h}_1\not=0$ or $h_1\bullet \hat{h}_2\not=0$, we are done.
Otherwise $\hat{h}=\hat{h}_1 + \hat{h}_2$ satisfies the required condition.

Returning to the proof, if $h_1\in H_L$ is nonzero, take $\hat{h} \in H_L$
such that $h\bullet \hat{h} \neq 0$ and $h_1\bullet \hat{h}\neq 0$.
First apply the Corollary to the plane $\langle h, \hat{h}\rangle_L$, yielding
that $\hat{h}+w$ is an $(L,G)$-rational line, and then apply it to the plane
$\langle \hat{h}, h_1\rangle_L$, showing that $h_1 + w$ is an $(L,G)$-rational
line, as required.
\end{proof}

In the next lemma it is important that the characteristic of~$K$ is greater than~$2$.

\begin{lem}\label{lem:construction}
Let $H_L$ be an $(L,G)$-rational symplectic space. Let $h, \tilde{h} \in H_L$
different from zero and let $w, \tilde{w} \in H_K^\perp$ such that
$w \bullet \tilde{w}  \in L^\times$
and $h+w, \tilde{h} + \tilde{w} \in \calL(G)$.

Then $\langle w, \tilde{w} \rangle_L$ is an $(L,G)$-rational symplectic plane.
\end{lem}

\begin{proof}
By Corollary~\ref{cor:translate} we have that $\langle h + \tilde{w}\rangle_L$
is an $(L,G)$-rational line.
As $(h+w)\bullet (h+\tilde{w}) = w \bullet \tilde{w} \in L^\times$, by Corollary~\ref{cor:mitchell} it follows
that $\langle w - \tilde{w} \rangle_L$ is an $(L,G)$-rational line.
Since $\langle -h-w\rangle_K\in \calL(G)$, by Corollary \ref{cor:translate} we have that $\langle -h+w\rangle_L$ is $(L, G)$-rational, and from $(-h+w)\bullet (h+\tilde{w}) = w \bullet \tilde{w} \in L^\times$
we conclude that $\langle w + \tilde{w} \rangle_L$ is an $(L,G)$-rational line.
As $(w - \tilde{w})\bullet (w + \tilde{w}) = 2 w \bullet \tilde{w} \in L^\times$,
we obtain that $\langle w + \tilde{w}, w - \tilde{w}\rangle_L = \langle w, \tilde{w} \rangle_L$
is an $(L,G)$-rational symplectic plane, as claimed.
\end{proof}

We now deduce that linking is an equivalence relation between mutually orthogonal spaces.
Note that reflexivity and symmetry are clear and only transitivity need be shown.

\begin{lem}\label{lem:equivrel}
Let $H_L$, $I_L$ and $J_L$ be mutually orthogonal $(L,G)$-rational symplectic
subspaces of~$V$.

If $H_L$ and $I_L$ are $(L,G)$-linked and also $I_L$ and $J_L$ are $(L,G)$-linked,
then so are $H_L$ and $J_L$.
\end{lem}

\begin{proof}
By definition there exist nonzero $h_0 \in H_L$, $i_0,i_1 \in I_L$ and $j_0 \in J_L$
such that $h_0 + i_0 \in \calL(G)$ and $i_1 + j_0 \in \calL(G)$.
There are $\hat{h}_0 \in H_L$ and $\hat{i}_0 \in I_L$
such that $\hat{h}_0 \bullet h_0 = 1$ and $\hat{i}_0 \bullet i_0 = 1$.

By Corollary~\ref{cor:translate} we have, in particular, that
$\langle h_0 + i_0 \rangle_L$,
$\langle \hat{i}_0 + j_0\rangle_L$ and $\langle \hat{h}_0 + (i_0 + \hat{i}_0)\rangle_L$
are $(L,G)$-rational lines.
As $(h_0 + \hat{i}_0)\bullet(i_0 + j_0) = 1$, by Corollary~\ref{cor:mitchell}
also $\langle h_0 + (i_0 + \hat{i}_0) + j_0\rangle_L$ is $(L,G)$-rational.
Furthermore, due to
$(\hat{h}_0 + (i_0 + \hat{i}_0)) \bullet (h_0 + (i_0 + \hat{i}_0) + j_0) = 1$,
it follows that $\langle (h_0 - \hat{h}_0) + j_0 \rangle_L$ is $(L,G)$-rational,
whence $H_L$ and $J_L$ are $(L,G)$-linked.
\end{proof}

\section{Merging linked orthogonal $(L,G)$-rational symplectic subspaces}\label{sec:merge}

We continue using our assumptions: $K$ is a finite field of characteristic at least~$5$,
$L \subseteq K$ a subfield, $V$ a $n$-dimensional symplectic $K$-vector space,
$G \subseteq \GSp(V)$ a subgroup.
In the previous section we established the existence of $(L,G)$-rational symplectic
planes in many cases (after allowing a conjugation of $G$ inside $\GSp(V)$).
In this section we aim at merging $(L,G)$-linked $(L,G)$-rational symplectic planes
into $(L,G)$-rational symplectic subspaces.

It is important to remark that no new conjugation of $G$ is required. The only
conjugation that is needed is the one from the previous section in order to have
an $(L,G)$-rational plane to start from.

\begin{lem}\label{lem:generate}
Let $H_L$ and $I_L$ be two $(L,G)$-rational symplectic subspaces of~$V$
which are $(L,G)$-linked. Suppose that $H_L$ and $I_L$ are orthogonal to each other. Then all lines in $H_L \oplus I_L$ are $(L,G)$-rational.
\end{lem}

\begin{proof}
The $(L,G)$-linkage implies the existence of $h_1 \in H_L$ and $w_1 \in I_L$
such that $\langle h_1+w_1 \rangle_K \in \calL(G)$.
By Corollary~\ref{cor:translate} $\langle h+w_1 \rangle_L$ is an $(L,G)$-rational line
for all $h \in H_L$. The same reasoning now gives that $\langle h+w \rangle_L$
is an $(L,G)$-rational line for all $h \in H_L$ and all $w \in I_L$.
\end{proof}

In view of Lemma~\ref{lem:definition} the above is (\ref{lem:definition:ii:a}). In order to
obtain (\ref{lem:definition:ii:b}), we need to invoke a result of Wagner. To make the exposition self-contained, we provide a proof in Appendix \ref{appendix:A}.

\begin{prop}\label{prop:Wagner}
Let $V$ be a $3$-dimensional vector space over a finite field $K$ of
characteristic $\ell\geq 5$, and let $G\subseteq \SL(V)$ be a group of
transformations fixing a $1$-dimensional vector space $U$. Let $U_1,
U_2, U_3$ be three distinct centres of transvections in $G$ such
that $U\not \subseteq U_1\oplus U_2$ and $U\not=U_3$. Then $(U_1 \oplus
U_2 ) \cap (U \oplus U_3)$ is the centre of a transvection of $G$.
\end{prop}

\begin{prop}\label{prop:intersect}
Let $U_1, U_2, U_3 \in \calL(G)$ and $W = U_1+U_2+U_3$. Assume $\dim W = 3$, $U_1$ and $U_2$ not orthogonal and let $U$ be a line in $W \cap W^\perp$ which is linearly independent from~$U_3$ and is not contained in $U_1\oplus U_2$.
Then $(U_1 \oplus U_2) \cap (U \oplus U_3)$ is a line in $\calL(G)$.
\end{prop}

\begin{proof}
Fix transvections $T_i\in G$ with centre $U_i$, $i=1, 2, 3$. These transvections fix $W$; let $H\subseteq \SL(W)$ be the group generated by the restrictions of the $T_i$ to $W$.  The condition $U \subseteq W^\perp$ guarantees that the $T_i$ fix~$U$ pointwise. Note that furthermore $U\not=U_3$ and $U\not\subseteq U_1\oplus U_2$. We can apply Proposition \ref{prop:Wagner}, and conclude that $(U_1\oplus U_2)\cap (U\oplus U_3)$ is the centre of a transvection $T$ of $H$. This transvection fixes the symplectic plane $U_1\oplus U_2$. Call $T_0$ the restriction of $T$ to this plane. It is a nontrivial transvection (since no line of $U_1\oplus U_2$ can be orthogonal to all $U_1\oplus U_2$). Hence by Lemma \ref{lem:extend} the line $(U_1\oplus U_2)\cap (U\oplus U_3)$ belongs to $\calL(G)$.
\end{proof}

We now deduce rationality statements from it.

\begin{cor}\label{cor:intersect}
Let $H_L$ be an $(L,G)$-rational symplectic plane and $U_3$ and $U_4$ be linearly
independent lines not contained in $H_K$.
Assume $U_4\subseteq H_K \oplus U_3$ is orthogonal to $H_K$ and to $U_3$
and assume that $U_3 \in \calL(G)$.

Then the intersection $H_K \cap (U_3 \oplus U_4) = I_K$ for some
line $I_L \subseteq H_L$.
\end{cor}

\begin{proof}
Choose two $(L,G)$-rational lines $U_{1,L}$ and $U_{2,L}$ such that
$H_L = U_{1,L} \oplus U_{2,L}$. With $U = U_4$ we can apply Proposition~\ref{prop:intersect}
in order to obtain that $I := H_K \cap (U_3 \oplus U_4)$ is a line in~$\calL(G)$
contained in~$H_K$.
As $H_L$ is $(L,G)$-rational, it follows that $I$ is $(L,G)$-rationalisable.
\end{proof}

\begin{cor}\label{cor:intersect-app}
Let $H_L \subseteq V$ be an $(L,G)$-rational symplectic space.
Let $h + w \in \calL(G)$ with $0 \neq h \in H_K$ and $w \in H_K^\perp$.
Then $h \in \calL(G)$.
In particular, $\langle h \rangle_K$ is an $(L,G)$-rationalisable line, i.e.\
there is $\mu \in K^\times$ such that $\mu h \in H_L$.
\end{cor}

\begin{proof}
If necessary replacing $H_L$ by any $(L,G)$-rational plane contained in~$H_L$,
we may without loss of generality assume that $H_L$ is an $(L,G)$-rational plane.
Let $y := h + w$. If $w = 0$, the claim follows from the $(L,G)$-rationality of~$H_L$.
Hence, we suppose $w \neq 0$.
Then $U_3 := \langle y \rangle_K$ is not contained in $H_K$.
Note that $w$ is perpendicular to $U_3$ and to $H_K$, and $w\in H_k\oplus \langle  y\rangle_K$. Hence, Corollary~\ref{cor:intersect}
gives that the intersection $H_K \cap (U_3 \oplus \langle w\rangle_K) = \langle h \rangle_K$ is in~$\calL(G)$.
\end{proof}

Corollary~\ref{cor:intersect-app} gives the rationalisability of a line.
In order to actually find a direction vector for a parameter in~$L$,
we need something extra to rigidify the situation. For this, we now take
a second link which is sufficiently different from the first link.

\begin{cor}\label{cor:rigid}
Let $H_L \subseteq V$ be an $(L,G)$-rational symplectic space.
Let $0 \neq \tilde{h} \in H_K$ and $\tilde{w} \in H_K^\perp$ such
that $\tilde{h}+\tilde{w} \in \calL(G)$.
Suppose that there are nonzero $h \in H_L$ and $w \in H_K^\perp$
such that $h+w \in \calL(G)$ and $w \bullet \tilde{w} \in L^\times$.

Then $\tilde{h} \in H_L$.
\end{cor}

\begin{proof}
By Corollary~\ref{cor:intersect-app} there is some $\beta \in K^\times$
such that $\beta \tilde{h} \in H_L$. We want to show $\beta \in L$.
By Corollary~\ref{cor:translate} we may assume that $h\bullet \tilde{h}\neq 0$, more precisely, $h \bullet (\beta \tilde{h}) = 1$; and we have furthermore that $\langle h + w \rangle_L$
is an $(L,G)$-rational line.
By Corollary~\ref{cor:mitchell}~(\ref{cor:mitchell:b}), $\langle h, \beta \tilde{h} \rangle_L$
is an $(L,G)$-rational symplectic plane contained in~$H_L$.
Let $c := w \bullet \tilde{w} \in L^\times$. We have
$$(h+w)\bullet(\tilde{h} + \tilde{w})
= h\bullet \tilde{h} + w \bullet \tilde{w}
= \frac{1}{\beta} + c =: \mu.$$
If $\mu = 0$, then $\beta \in L$ and we are done. Assume $\mu \neq 0$.
By Corollary~\ref{cor:mitchell}~(\ref{cor:mitchell:b}) it follows that
$\langle h+w, \mu^{-1}(\tilde{h}+\tilde{w}) \rangle_L$
is an $(L,G)$-rational symplectic plane.
Thus, $\langle h+w + \mu^{-1}(\tilde{h}+\tilde{w}) \rangle_L$
is an $(L,G)$-rational line.
By Corollary~\ref{cor:intersect-app} there is some $\nu \in K^\times$
such that $\nu(h + \mu^{-1} \tilde{h}) \in H_L.$
Consequently, $\nu \in L^\times$, whence $\mu\in L$, so that $\beta \in L$.
\end{proof}

The main result of this section is the following merging result.

\begin{prop}\label{prop:link}
Let $H_L$ and $I_L$ be orthogonal $(L,G)$-rational symplectic subspaces of~$V$
that are $(L,G)$-linked.

Then $H_L \oplus I_L$ is an $(L,G)$-rational symplectic subspace of~$V$.
\end{prop}

\begin{proof}
We use Lemma~\ref{lem:definition}. Part (\ref{lem:definition:ii:a}) follows directly from
Lemma~\ref{lem:generate}. We now show (\ref{lem:definition:ii:b}). Let $h+w \in \calL(G)$
with nonzero $h \in H_K$ and $w \in I_K$ be given.
Corollary~\ref{cor:intersect-app} yields $\mu, \nu \in K^\times$
such that $\mu h \in H_L$ and $\nu w \in I_L$.
Let $\hat{h} \in H_L$ with $(\mu h) \bullet \hat{h} = 1$, as well as
$\hat{w} \in I_L$ with $(\nu w) \bullet \hat{w} = 1$.
Lemma~\ref{lem:generate} tells us that $\hat{h} + \hat{w} \in \calL(G)$.
Together with $(\nu h) + (\nu w) \in \calL(G)$,
Corollary~\ref{cor:rigid} yields $\nu h \in H_L$, whence
$\nu h + \nu w\in H_L \oplus I_L$.
\end{proof}

\section{Extending $(L,G)$-rational spaces}\label{sec:extend}

We continue using the same notation as in the previous sections.
Here, we will use the merging results in order to extend $(L,G)$-rational
symplectic spaces.

\begin{prop}\label{prop:construction}
Let $H_L$ be a nonzero $(L,G)$-rational symplectic subspace of~$V$.
Let nonzero $h,\tilde{h} \in H_K$, $w,\tilde{w}\in H_K^\perp$ be such that
$h+w,\tilde{h}+\tilde{w} \in \calL(G)$ and $w \bullet \tilde{w} \neq 0$.

Then there exist $\alpha,\beta \in K^\times$ such that
$\langle \alpha w,\beta \tilde{w}\rangle_L$
is an $(L,G)$-rational symplectic plane which is $(L,G)$-linked with~$H_L$.
\end{prop}

\begin{proof}
By Corollary~\ref{cor:intersect-app} we may and do assume by scaling
$h+w$ that $h \in H_L$. Furthermore, we assume by scaling
$\tilde{h}+\tilde{w}$ that $w \bullet \tilde{w} = 1$.
Then Corollary~\ref{cor:rigid} yields that $\tilde{h} \in H_L$.
We may appeal to Lemma~\ref{lem:construction} yielding that
$\langle w, \tilde{w} \rangle_L$ is an $(L,G)$-rational plane.
The $(L,G)$-link is just given by $h+w$.
\end{proof}

\begin{cor}\label{cor:construction}
Let $H_L$ be a non-zero $(L,G)$-rational symplectic subspace of~$V$.
Let nonzero $h,\tilde{h} \in H_K$, $w,\tilde{w}\in H_K^\perp$ be such that
$h+w,\tilde{h}+\tilde{w} \in \calL(G)$ and $w \bullet \tilde{w} \neq 0$.

Then there is an $(L,G)$-rational symplectic subspace $I_L$ of~$V$
containing $H_L$ and such that $I_K = \langle H_K, w, \tilde{w} \rangle_K$.
\end{cor}

\begin{proof}
This follows directly from Propositions~\ref{prop:construction} and~\ref{prop:link}.
\end{proof}

\begin{prop}\label{prop:step}
Assume $\langle \calL(G) \rangle_K = V$.
Let $H_L$ be a nonzero $(L,G)$-rational symplectic space.
Let $0 \neq v \in \calL(G) \setminus (H_K \cup H_K^\perp)$.

Then there is an $(L,G)$-rational symplectic space~$I_L$ containing $H_L$
such that $v \in I_K$.
\end{prop}

\begin{proof}
We write $v = h + w$ with $h \in H_K$ and $w \in H_K^\perp$. Note that both
$h$ and $w$ are nonzero by assumption.
As $\langle \calL(G) \rangle_K = V$, we may choose $\tilde{v} \in \calL(G)$
such that $\tilde{v} \bullet w \neq 0$.
We again write $\tilde{v} = \tilde{h} + \tilde{w}$
with $\tilde{h} \in H_K$ and $\tilde{w} \in H_K^\perp$.

We, moreover, want to ensure that $\tilde{h} \neq 0$.
If $\tilde{h} = 0$, then we proceed as follows. Corollary~\ref{cor:intersect-app}
implies the existence of $\mu \in K^\times$ such that $\mu h \in H_L$.
Now replace $h$ by $\mu h$ and $w$ be $\mu w$. Then Corollary~\ref{cor:translate}
ensures that $\langle h+w \rangle_L$ is an $(L,G)$-rational line.
Furthermore, scale $\tilde{w}$ so that $(h+w) \bullet \tilde{w} \in L^\times$, whence by
Corollary~\ref{cor:mitchell} $h+w+\tilde{w} \in \calL(G)$. We use this
element as $\tilde{v}$ instead.
Note that it still satisfies $\tilde{v} \bullet w \neq 0$, but now $\tilde{h} \neq 0$.

Now we are done by Corollary~\ref{cor:construction}.
\end{proof}

\begin{cor}\label{cor:twothree}
Assume $\langle \calL(G) \rangle_K = V$, and let $H_L$ be an $(L,G)$-rational symplectic space.

Then there is an $(L,G)$-rational symplectic space~$I_L$ containing~$H_L$ such that $\calL(G) \subseteq I_K \cup I_K^\perp$.
\end{cor}

\begin{proof}
Iterate Proposition~\ref{prop:step}.
\end{proof}

\section{Proof of Theorem~\ref{thm:gp}}\label{sec:proof}

In this section we will finish the proof of Theorem~\ref{thm:gp}.

\begin{lem}\label{lem:dec}
Let $V = S_1 \oplus \dots \oplus S_h$ be a decomposition of~$V$ into
linearly independent, mutually orthogonal subspaces such that
$\calL(G) \subseteq S_1 \cup \dots \cup S_h$.

\begin{enumerate}[(a)]
\item\label{lem:dec:a} If $v_1,v_2 \in \calL(G) \cap S_1$ are such that $v_1+v_2 \in \calL(G)$,
then for all $g \in G$ there exists an index $i \in \{1,\dots,h\}$ such that $g(v_1)$ and
$g(v_2)$ belong to the same~$S_i$.
\item\label{lem:dec:b} If $S_1$ is $(L,G)$-rationalisable, then for all $g \in G$ there exists an
index $i \in \{1,\dots,h\}$ such that $gS_1 \subseteq S_i$.
\end{enumerate}
\end{lem}

\begin{proof}
(\ref{lem:dec:a}) Assume that $g(v_1) \in S_i$ and $g(v_2) \in S_j$ with $i \neq j$. Then
$g(v_1)+g(v_2) = g(v_1 + v_2) \in \calL(G)$ satisfies $g(v_1+v_2) \in S_i \oplus S_j$,
but it neither belongs to~$S_i$ nor to~$S_j$.
This contradicts the assumption that $\calL(G) \subseteq S_1 \cup \dots \cup S_h$.

(\ref{lem:dec:b}) If $S_1 = S_{1,L}$ with $S_{1,L}$ an $(L,G)$-rational space, we can
apply~(\ref{lem:dec:a}) to an $L$-basis of~$S_{1,L}$.
\end{proof}

\begin{cor}\label{cor:dec}
Let $I_L \subseteq V$ be an $(L,G)$-rational symplectic subspace such that
$\calL(G) \subseteq I_K \cup I_K^\perp$ and let $g \in G$.
Then either $g(I_K) = I_K$ or $g(I_K) \subseteq I_K^\perp$;
in the latter case $I_K \cap g(I_K) = 0$.
\end{cor}

\begin{proof}
This follows from Lemma~\ref{lem:dec} with $S_1 = I_K$ and $S_2 = I_K^\perp$.
\end{proof}

We are now ready to prove Theorem~\ref{thm:gp}.

\begin{proof}[Proof of Theorem~\ref{thm:gp}.]
As we assume that $G$ contains some transvection, it follows that $\calL(G)$ is
nonempty and consequently $\langle \calL(G) \rangle_K$ is a nonzero $K$-vector
space stabilised by~$G$ due to Lemma~\ref{lem:fix}.
Hence, either we are in case~\ref{thm:gp:1} of Theorem~\ref{thm:gp}
or $\langle \calL(G) \rangle_K = V$, which we assume now.

From Proposition~\ref{prop:mitchell} we obtain that there is some
$A \in \GSp(V)$, a subfield $L \le K$ such that there is an $(L,AGA^{-1})$-rational
symplectic plane~$H_{L}$.
Since the statements of Theorem~\ref{thm:gp} are not affected by this conjugation,
we may now assume that $H_{L}$ is $(L,G)$-rational.

From Corollary~\ref{cor:twothree} we obtain an $(L,G)$-rational
symplectic space $I_{1,L}$ such that $\calL(G) \subseteq I_{1,K}\cup I_{1,K}^\perp$.
If $I_{1,K} = V$, then we know due to $I_{1,L} \cong L^{n}$
that $G$ contains a transvection whose direction is
any vector of~$I_{1,L}$. As the transvections generate the symplectic group, it
follows that $G$ contains $\Sp(I_{1,L})\cong \Sp_{n}(L)$
and we are in case~\ref{thm:gp:3} of Theorem~\ref{thm:gp}.
Hence, suppose now that $I_{1,K} \neq V$.

Either every $g\in G$ stabilises $I_{1, K}$, and we are in case~\ref{thm:gp:1} and done,
or there is $g \in G$ and $v \in I_{1,L}$ with $g(v) \not\in I_{1,K}$.
Set $I_{2, L}:=gI_{1, L}$.
Note that $I_{2,L} \subseteq \calL(G)$ because of Lemma~\ref{lem:basechange}.
Now we apply Corollary~\ref{cor:dec} to the decomposition $V=I_{1, K}\oplus I_{1, K}^{\perp}$
and obtain that $g(I_{1, K})\subseteq I_{1, K}^{\perp}$.
Moreover $\calL(G)=\calL(gGg^{-1})\subseteq gI_{1, K}\cup gI_{1, K}^{\perp}=I_{2, K}\cup I_{2, K}^{\perp}$.

We now have
$\calL(G) \subseteq I_{1, K}\cup I_{2, K}\cup (I_{1, K}\oplus I_{2, K})^{\perp}$.
Either $I_{1, K}\oplus I_{2, K}=V$ and $(I_{1, K}\oplus I_{2, K})^{\perp}=0$,
or there are two possibilities:

\begin{itemize}
\item  For all $g\in G$, $gI_{1, L}\subseteq I_{1, K}\cup I_{2, K}$.
If this is the case, then $G$ fixes the space $I_{1, K}\oplus I_{2, K}$,
and we are in case~\ref{thm:gp:1}, and done.

\item There exists $g\in G$, $v\in I_{1, L}$ such that $g(v)\not\in I_{1, K}\cup I_{2, K}$.
Set $I_{3, L}=gI_{1, L}$. Due to $\calL(G)\subseteq I_{3, K}\cup I_{3, K}^{\perp}$, we then have
$\calL(G) \subseteq I_{1, K}\cup I_{2, K}\cup I_{3, K}\cup(I_{1, K}\oplus I_{2, K}\oplus I_{3, K})^{\perp}$.
\end{itemize}

Hence, iterating this procedure, we see that either we are in case~\ref{thm:gp:1},
or we obtain a decomposition $V=I_{1, K}\oplus\cdots \oplus I_{h, K}$
with mutually orthogonal symplectic spaces such that
$\calL(G)\subseteq I_{1, K}\cup \cdots \cup I_{h, K}$.

Note that Lemma~\ref{lem:dec} implies that $G$ respects this decomposition
in the sense that for all $i\in \{1,\dots,h\}$ there is $j\in \{1,\dots,h\}$
such that $g(I_{i,K})=I_{j,K}$. If the resulting action of $G$ on the index
set $\{1,\dots,h\}$ is not transitive, then we are again in case~\ref{thm:gp:1},
otherwise in case~\ref{thm:gp:2}.
\end{proof}

\appendix
\section{A result on transvections in a 3-dimensional vector space}\label{appendix:A}

In this appendix we provide a proof of the following result concerning subgroups in a $3$-dimensional vector space that was used in Section~\ref{sec:merge}:

\begin{prop}\label{thm:Wagner}
Let $V$ be a $3$-dimensional vector space over a finite field $K$ of
characteristic $\ell\geq 5$, and let $G\subseteq \SL(V)$ be a subgroup satisfying:
\begin{enumerate}
 \item There exists a $1$-dimensional $K$-vector space $U$ such that $G\vert_U=\{\id_U\}$.
 \item There exist $U_1, U_2, U_3$ three distinct centres of transvections in $G$ such
that $U\not\subseteq U_1\oplus U_2$ and $U\not=U_3$.
\end{enumerate}
Then $(U_1 \oplus U_2 ) \cap (U \oplus U_3)$ is the centre of a transvection of $G$. 
\end{prop}

This result is Theorem 3.1(a) of~\cite{Wa}. Below we have written the proof in detail.  We will essentially follow the original proof of Wagner~\cite{Wa}, reformulating it with the terminology developed in this paper. We follow~\cite{Mitchell1911} when Wagner refers to the results proven there.
We also used \cite{Mit} to `get a feeling' of the ideas used in~\cite{Wa}.

The setting differs from that of the rest of the paper, since there is no symplectic structure. One consequence of this is that the axis of a transvection $\tau$ in $\SL(V)$ is not determined by its centre. Given any plane $W\subset V$ and any line $U\subset V$, there exist transvections with axis $W$ and centre~$U$;
namely, fixing an element $\varphi\in \Hom(V, K)=V^*$ of the dual vector space of $V$ such that
$W=\ker(\varphi)$, and fixing a nonzero vector $u\in U$, then all transvections in $\SL(V)$ with axis $W$ and centre~$U$ are given by
\begin{equation*}
 \tau(v):= v + \lambda \varphi(v) u
\end{equation*}
for some $\lambda\in K$ (cf.\ \cite{Ar}, p.~160).

A key input in the proof is Lemma~\ref{lem:1911}. In order to apply it to a subplane $W\subset V$, we need to endow it with some symplectic structure. We do so by choosing any two linearly independent vectors $e_1$, $e_2$ and considering the symplectic structure defined by declaring $\{e_1, e_2\}$ to be a symplectic basis.

\begin{proof}[Proof of Proposition~\ref{thm:Wagner}]
Without loss of generality we may assume that $G$ is generated by
transvections. In particular, we may assume $G\subseteq \SL(V)$.

The hypotheses imply that the inclusion $U_3\subseteq (U_1\oplus U)\cap (U_2\oplus U)$ does not hold.  Indeed, assume $U_3\subseteq (U_1\oplus U)\cap (U_2\oplus U)$. We know that $V=U_1\oplus U_2\oplus U$, hence $U_1\oplus U\not=U_2\oplus U$, so that $(U_1\oplus U)\cap (U_2\oplus U)$ has dimension $1$. Therefore $U_3=(U_1\oplus U)\cap (U_2\oplus U)=U$, but by hypothesis $U_3\not=U$. Interchanging $U_1$ and $U_2$ if necessary we can assume that $U_3\not\subseteq U_1\oplus U$.

For $i=2, 3$, let $W_{1, i}=U_1\oplus U_i$ and $G_{1, i}$ be the subgroup of $\GL(W_{1, i})$ generated by the transvections in $G$ that preserve the plane $W_{1, i}$.
We want to endow $W_{1, i}$ with a suitable $(L, G_{1, i})$-rational structure.
In particular, we want that these structures are compatible.

For each $i=1, 2, 3$, fix
a transvection $T_i\in G$ with centre $U_i$. 
Note that, since $G\vert_U$ is the identity
and $U\not=U_i$, the axis of $T_i$ (that is, the
plane pointwise fixed by it) must be $U_i\oplus U$. 

The transvections $T_1$ and $T_2$ preserve the plane $U_1\oplus U_2$, and
since this plane does not coincide with the axis of $T_1$ or $T_2$,
they both act as nontrivial transvections on $U_1 \oplus U_2$. We apply Lemma \ref{lem:1911}  to the $2$-dimensional $K$-vector space $W_{1, 2}$ (which we endow with a symplectic structure with symplectic basis $\{u_1, u_2\}$ such that $u_1\in U_1$ and $u_2\in U_2$) and the group $G_{1, 2}$ and obtain a matrix $A\in \GL_2(K)$ such that $AU_1=U_1$, $AU_2=U_2$ and a subfield $L$ of $K$ such that $(W_{1, 2})_L$ is an $(L, AG_{1, 2}A^{-1})$-rational plane. Since $U$ is linearly independent from $U_1\oplus U_2$, we can extend $A$ to an element of $\GL(V)$ such that $AU=U$. Without loss of generality we can replace $G$  by $AGA^{-1}$ and $U_3$ by $AU_3$. Thus $(W_{1, 2})_L=\langle u_1, u_2\rangle_L$ is an $(L, G_{1, 2})$-rational plane. 

Since $V=U_1\oplus U_2\oplus U$, we find $a_1,a_2 \in K$ such that $0 \neq u+a_1u_1+a_2u_2 \in U_3$
with some $u \in U$. By hypothesis $a_2\neq 0$. Hence by normalising, we can assume
$0 \neq u_3 := -u+a_1u_1+u_2 \in U_3$, so that we have the relation
\begin{equation}\label{eq:rel}
 u=a_1 u_1 + u_2 + u_3.
\end{equation}
The set $\mathcal{B} = \{u_1,u_2,u\}$ is a $K$-basis of~$V$.
The proof will be finished if we show that $G$ contains a transvection of direction
$u_3-u=-a_1u_1 -u_2\in (U\oplus U_3)\cap (U_1\oplus U_2)$.

Now we consider the plane $W_{1, 3}$, and endow it with a symplectic structure with symplectic basis $\{u_1, u_3\}$. We claim that $\langle u_1, u_3\rangle_L$ is an $(L, G_{1, 3})$-rational plane. Indeed, if we show that $\langle u_1\rangle_L$ is an $(L, G_{1, 3})$-rational line, then Corollary \ref{cor:mitchell}(\ref{cor:mitchell:b}) applied to $U_{1, L}=\langle u_1\rangle_L$ and $U_3$ (which lies in $\mathcal{L}(G_{1,3})$ because by hypothesis $G$ contains a transvection with centre $U_3$) yields the result.
Consider the set of transvections of $G$ with centre $U_1$. As discussed above, their axis is $U\oplus U_1=\{v\in V: p_2(v)=0\}$, where $p_2$ denotes the projection in the second coordinate with respect to the basis~$\mathcal{B}$. Thus any transvection of $G$ with direction $U_1$ can be written as $T_1(v)=v + \lambda p_2(v) u_1$ for some $\lambda\in K$. Restricting $T_1$ to $W_{1, 2}$, and taking into account that $p_2(v)=-v\bullet u_1$ with $v\in W_{1,2}$ for the symplectic structure on $W_{1, 2}$ with symplectic basis $\{u_1, u_2\}$, it follows from the
$(L, G_{1, 2})$-rationality of $\langle u_1, u_2\rangle_L$ that $\lambda\in L$. Now we restrict to $W_{1, 3}$. Note that $p_2(v)= v\bullet u_1$ for $v \in W_{1,3}$, where $\bullet$ denotes the symplectic structure on $W_{1, 3}$ defined by the symplectic basis $\{u_1, u_3\}$. Thus the restriction of $T_1$ to $W_{1, 3}$ is $T_1(v)=v + \lambda (v\bullet u_1 )u_1$. This proves the $(L, G_{1, 3})$-rationality of $\langle u_1\rangle_L$.

The discussion above shows that, if we fix the basis $\{u_1, u_i\}$ of $W_{1, i}$, then $G_{1, i}$ contains $\SL_2(L)$; in particular it contains the reflection given by $(u_1\mapsto -u_1, u_i\mapsto -u_i)$. Since $G$ acts as the identity on $U$, we obtain that $G$ contains the element $\delta_{1, i}$ given by $(u_1\mapsto -u_1, u_i\mapsto -u_i, u\mapsto u)$.
With respect to the basis $\mathcal{B}$, these elements have the shape $\delta_{1, 2}=\begin{pmatrix}-1 & 0 & 0\\
   0 & -1 & 0\\
   0 & 0 & 1
\end{pmatrix}$ and  $\delta_{1, 3}=\begin{pmatrix}-1 & 0 & 0\\
   0 & -1 & 0\\
   0 & 2 & 1
\end{pmatrix}$. Thus 
$T:=\delta_{1, 2}\delta_{1, 3}=\begin{pmatrix}1 & 0 & 0\\
   0 & 1 & 0\\
   0 & 2 & 1\end{pmatrix}$ is a transvection of centre $U$ and axis $U\oplus U_1$. Since $2$ is invertible in $\mathbb{F}_{\ell}$, we can find $k\in \mathbb{Z}$ such that 
$T^k=\begin{pmatrix} 1 & 0 & 0\\ 0 & 1 & 0\\ 0 & 1 & 1\end{pmatrix}$.
The transvection $T^{k}\circ T_3\circ T^{-k}\in G$ has direction  $T^{k}(u_3)=u_3-u$; this is the transvection we were seeking.

\end{proof}

\bibliography{Bibliog}

\begin{thebibliography}{AdDW13b}

\bibitem[AdDSW13]{partIII}
Sara Arias-de{-}Reyna, Luis Dieulefait, Sug~Woo Shin, and Gabor Wiese.
\newblock Compatible systems of symplectic {G}alois representations and the
  inverse {G}alois problem {III}. {A}utomorphic construction of compatible
  systems with suitable local properties.
\newblock {\em Preprint}, 2013.

\bibitem[AdDW13a]{partI}
Sara Arias-de{-}Reyna, Luis Dieulefait, and Gabor Wiese.
\newblock Compatible systems of symplectic {G}alois representations and the
  inverse {G}alois problem {I}. {I}mages of projective representations.
\newblock {\em Preprint, arXiv:1203.6546}, 2013.

\bibitem[AdDW13b]{partII}
Sara Arias-de{-}Reyna, Luis Dieulefait, and Gabor Wiese.
\newblock Compatible systems of symplectic {G}alois representations and the
  inverse {G}alois problem {II}. {T}ransvections and huge image.
\newblock {\em Preprint, arXiv:1203.6552}, 2013.

\bibitem[Art57]{Ar}
E.~Artin.
\newblock {\em Geometric algebra}.
\newblock Interscience Publishers, Inc., New York-London, 1957.

\bibitem[Dic58]{Dickson}
Leonard~Eugene Dickson.
\newblock {\em Linear groups: {W}ith an exposition of the {G}alois field
  theory}.
\newblock with an introduction by W. Magnus. Dover Publications Inc., New York,
  1958.

\bibitem[Kan79]{Kantor1979}
William~M. Kantor.
\newblock Subgroups of classical groups generated by long root elements.
\newblock {\em Trans. Amer. Math. Soc.}, 248(2):347--379, 1979.

\bibitem[LZ82]{LZ}
Shang~Zhi Li and Jian~Guo Zha.
\newblock On certain classes of maximal subgroups in {${\rm PSp}(2n,\,F)$}.
\newblock {\em Sci. Sinica Ser. A}, 25(12):1250--1257, 1982.

\bibitem[Mit11]{Mitchell1911}
Howard~H. Mitchell.
\newblock Determination of the ordinary and modular ternary linear groups.
\newblock {\em Trans. Amer. Math. Soc.}, 12(2):207--242, 1911.

\bibitem[Mit14]{Mit}
Howard~H. Mitchell.
\newblock The subgroups of the quaternary abelian linear group.
\newblock {\em Trans. Amer. Math. Soc.}, 15(4):379--396, 1914.

\bibitem[Wag74]{Wa}
Ascher Wagner.
\newblock Groups generated by elations.
\newblock {\em Abh. Math. Sem. Univ. Hamburg}, 41:190--205, 1974.

\end{thebibliography}
\bibliographystyle{alpha}

\end{document}